\documentclass[12pt,a4paper]{amsart}
\usepackage{amsfonts}
\usepackage{amsthm}
\usepackage{amsmath}
\usepackage{amscd}
\usepackage[latin2]{inputenc}
\usepackage{t1enc}
\usepackage[mathscr]{eucal}
\usepackage{indentfirst}
\usepackage{graphicx}
\usepackage{graphics}
\usepackage{pict2e}
\usepackage{epic}
\numberwithin{equation}{section}
\usepackage[margin=2.9cm]{geometry}

\theoremstyle{plain}
\newtheorem{Th}{Theorem}[section]
\newtheorem{Lemma}[Th]{Lemma}
\newtheorem{Cor}[Th]{Corollary}
\newtheorem{Prop}[Th]{Proposition}

 \theoremstyle{definition}
\newtheorem{Def}[Th]{Definition}

\newtheorem{Rem}[Th]{Remark}
\newtheorem{?}[Th]{Problem}
\newtheorem{Ex}[Th]{Example}

\newcommand{\G}{\mathbb{G}}
\newcommand{\prm}{\mathrm{pm}}

\begin{document}

\title{Matchings in vertex-transitive bipartite graphs}

\author[P. Csikv\'ari]{P\'{e}ter Csikv\'{a}ri}

\address{Massachusetts Institute of Technology \\ Department of Mathematics \\
Cambridge MA 02139 \&  E\"{o}tv\"{o}s Lor\'{a}nd University \\ Department of Computer 
Science \\ H-1117 Budapest
\\ P\'{a}zm\'{a}ny P\'{e}ter s\'{e}t\'{a}ny 1/C \\ Hungary} 

\email{peter.csikvari@gmail.com}

\thanks{The author  is partially supported by the
Hungarian National Foundation for Scientific Research (OTKA), grant
no. K81310 and K109684.}



\begin{abstract} A theorem of A. Schrijver asserts that a $d$--regular bipartite graph on $2n$ vertices has at least 
$$\left(\frac{(d-1)^{d-1}}{d^{d-2}}\right)^n$$
perfect matchings. L. Gurvits gave an extension of Schrijver's theorem for matchings of density $p$. In this paper we give a stronger version of Gurvits's theorem in the case of vertex-transitive bipartite graphs. This stronger version in particular implies that for every positive integer $k$, there exists a positive constant $c(k)$ such that if a $d$-regular vertex-transitive bipartite graph on $2n$ vertices contains a cycle of length at most $k$, then it has at least 
$$\left(\frac{(d-1)^{d-1}}{d^{d-2}}+c(k)\right)^n$$
perfect matchings.

We also show that if $(G_i)$ is a Benjamini--Schramm convergent graph sequence of vertex-transitive bipartite graphs, then
$$\frac{\ln \prm(G_i)}{v(G_i)}$$
is convergent, where $\prm(G)$ and $v(G)$ denote the number of perfect matchings and the number of vertices of $G$, respectively. 

We also show that if $G$ is  $d$--regular vertex-transitive bipartite graph on $2n$ vertices and $m_k(G)$ denotes the number of matchings of size $k$, and 
$$M(G,t)=1+m_1(G)t+m_2(G)t^2+\dots +m_n(G)t^n=\prod_{k=1}^n(1+\gamma_k(G)t),$$
where $\gamma_1(G)\leq \dots \leq \gamma_n(G)$, then
$$\gamma_k(G)\geq \frac{d^2}{4(d-1)}\frac{k^2}{n^2},$$
and
$$\frac{m_{n-1}(G)}{m_n(G)}\leq \frac{2}{d}n^2.$$
The latter result improves on a previous bound of C. Kenyon, D. Randall and A. Sinclair.
There are examples of $d$--regular bipartite graphs for which these statements fail to be true without the condition of vertex-transitivity.
\end{abstract}

\maketitle

\section{Introduction} 

This paper is motivated by two seemingly independent sets of results on perfect matchings of finite graphs.  The first set of results concerns with extremal values of the number of (perfect) matchings, most notably results of A. Schrijver and L. Gurvits stand as cornerstones. The second set of results deals with a convergent graph sequence $(G_i)$, and the 
$$\lim_{i\to\infty}\frac{\ln \prm(G_i)}{v(G_i)},$$
where $\prm(G)$ and $v(G)$ denote the number of perfect matchings, and the number of vertices of the graph $G$, respectively. Here the main question is that what kind of conditions we have to impose to the graphs $(G_i)$ and to the convergence in order to ensure the existence of the above limit. 

The remaining part of the Introduction is split into two parts according to the two topics. We note here that we use standard terminology, but in case of a concept undefined in the Introduction, the first paragraph of Section~\ref{entropy-function} might help.

\subsection{Extremal problems about the number of matchings in bipartite graphs.}
Here the starting point is the following theorem of A. Schrijver.

\begin{Th}[A. Schrijver \cite{sch1}, for $d=3$ M. Voorhoeve \cite{vor}] \label{Schrijver} Let $G$ be a $d$--regular bipartite graph on $2n$ vertices, and let $\prm(G)$ denote the number of perfect matchings of $G$. Then
$$\prm(G)\geq \left(\frac{(d-1)^{d-1}}{d^{d-2}}\right)^n.$$
\end{Th}
Note that Schrijver and Valiant proved in \cite{sch2} that the number
$$\frac{(d-1)^{d-1}}{d^{d-2}}$$
cannot be improved by showing that for a random $d$--regular bipartite
multigraph the statement is asymptotically tight. In \cite{ACFK} the authors proved that actually large girth graphs (not only random graphs) have asymptotically the same number of perfect matchings: let $g(H)$ denote the girth of a graph $H$, i. e., the length of the shortest cycle in $H$.
Then the following is true.

\begin{Th}[\cite{ACFK}] \label{entropy-girth} Let $(G_i)$ be a sequence of $d$--regular bipartite graphs such that $g(G_i)\to \infty$, where $g(G_i)$ denotes the girth of 
$G_i$. Then
$$\lim_{i\to \infty}\frac{\ln \prm(G_i)}{v(G_i)}=\frac{1}{2}\ln \left(
\frac{(d-1)^{d-1}}{d^{d-2}}\right).$$ 
\end{Th}
L. Gurvits gave an extension of Schrijver's theorem for matchings of size $k$: 

\begin{Th}[Gurvits \cite{gur2}] Let $G$ be an arbitrary $d$--regular bipartite graph on $v(G)=2n$ vertices. Let $m_k(G)$ denote the number of $k$--matchings.  Let $p=\frac{k}{n}$.
Then
$$\frac{\ln m_k(G)}{v(G)}\geq \frac{1}{2}\left(p\ln \left(\frac{d}{p}\right)+(d-p)\ln \left(1-\frac{p}{d}\right)-2(1-p)\ln (1-p)\right)+o_{v(G)}(1).$$
\end{Th}
It is worth introducing a notation for the function appearing in this inequality:
$$\G_d(p)=\frac{1}{2}\left(p\ln \left(\frac{d}{p}\right)+(d-p)\ln \left(1-\frac{p}{d}\right)-2(1-p)\ln (1-p)\right).$$

We note that Gurvits \cite{gur2} gave an effective form of this result, but for our purposes any $o_{v(G)}(1)$ term would suffice as we will use another form of this inequality where this term can be vanished. We also mention that in the current form of this inequality, it holds only for some special values of $p$. To achieve the aforementioned more convenient form of Gurvits's inequality, we will introduce the so-called entropy function $\lambda_G(p)$ in Section~\ref{entropy-function}.

For this function we have
$$\lambda_G(p)\approx \frac{\ln m_k(G)}{v(G)},$$
and Gurvits's theorem can be rewritten as
$$\lambda_G(p)\geq \G_d(p).$$
Moreover, we will also see that if $G$ contains a perfect matching, then
$$\lambda_G(1)=\frac{\ln \prm(G)}{v(G)}.$$
In Section~\ref{gurvits} we will prove the following extension of Gurvits's theorem for vertex-transitive bipartite graphs which also implies that the bound given in Theorem~\ref{Schrijver} can be improved for vertex-transitive bipartite graphs containing  short cycles:

\begin{Th} \label{stability} Let $G$ be a finite $d$--regular vertex-transitive bipartite graph, where $d\geq 2$. Furthermore, let the gap function $g(p)$ be defined as
$$g(p)=\lambda_G(p)-\G_d(p).$$
Then $g(p)$ is monotone increasing function with $g(0)=0$, in particular $g(p)$ is non-negative. Furthermore, if $G$ contains an $\ell$-cycle, then
$$g(p)\geq \int_0^pf(x)^{\ell}\, dx,$$
where
$$f(x)=\frac{1}{4d}\min(x,(1-x)^2).$$
\end{Th}

\subsection{The limit of  perfect matching entropies.} In statistical physics, the dimer model is one of the most studied model. One of its main problems is the following. Let $L$ be an infinite lattice, and let $(G_i)$ be a sequence of finite graphs exhausting  $L$. The problem is to find 
$$\lim_{i\to\infty}\frac{\ln \prm(G_i)}{v(G_i)}.$$
It turns out that the actual limit heavily depends on the exhaustion (and may not exist). The best known example if $(G_i)$ are larger and larger boxes of the infinite square grid $\mathbb{Z}^2$, then the celebrated result of Kasteleyn \cite{kas} and independently Temperley and Fisher \cite{tem} asserts that the limit is $G/\pi$, where $G$ is the Catalan constant. On the other hand, it turns out that if one considers the sequence of Aztec diamonds for $(G_i)$, then the limit is $(\ln 2)/4$ (see \cite{eklp}). This reflects the fact that the boundary of a graph can affect the number of perfect matchings. On the other hand, the situation is not as bad as it seems for the first sight: in \cite{ckp} H. Cohn, R. Kenyon and J. Propp showed how one can take into account the boundaries of the graphs. Another way to overcome the difficulty of the boundary is to consider doubly periodic graphs as it was done in \cite{kos} by R. Kenyon, A. Okounkov and S. Sheffield. They considered $\mathbb{Z}^2$--periodic bipartite planar graphs $L$, and $G_i$ was  the quotient of $L$ by the action of $(i\mathbb{Z})^2$. In this setting they were able to determine the limit explicitly as a certain integral. In both papers \cite{ckp} and \cite{kos}, the techniques heavily relied on the planarity of the graph $L$. 
\medskip

In this paper we present an abstract version of these results, where we are not confined to planar graphs. Then we need to introduce a convergence concept replacing the exhaustion of $L$. This concept is the Benjamini--Schramm convergence. With some foresight we also define the limit objects of Benjamini--Schramm convergent graph sequences, the so-called \emph{random rooted graphs}.

\begin{Def} Let $L$ be a probability distribution on (infinite) rooted graphs; we will call $L$ a \emph{random rooted graph}.
For a finite rooted graph $\alpha$ and a positive integer $r$, let $\mathbb{P}(L,\alpha,r)$ be the probability that the $r$-ball
centered at a random root vertex chosen from the distribution $L$ is isomorphic to $\alpha$.

For a finite graph $G$, a finite rooted graph $\alpha$ and a positive integer
$r$, let $\mathbb{P}(G,\alpha,r)$ be the probability that the $r$-ball
centered at a uniform random vertex of $G$ is isomorphic to $\alpha$. 

We say that a sequence $(G_n)$ of bounded degree graphs is \emph{Benjamini--Schramm
convergent} if for all finite rooted graphs $\alpha$ and $r>0$, the
probabilities $\mathbb{P}(G_n,\alpha,r)$ converge. Furthermore, we say that \emph{$(G_{n})$ Benjamini--Schramm converges to $L$},
if for all positive integers $r$ and finite rooted graphs $\alpha$, $\mathbb{P}(G_{n},\alpha,r)\rightarrow \mathbb{P}(L,\alpha,r)$. 
\end{Def}

\begin{Ex} Let us consider a sequence of boxes in $\mathbb{Z}^d$ where all sides converge to infinity. This will be Benjamini--Schramm convergent graph sequence since for every fixed $r$, we will pick a vertex which at least $r$-far from the boundary with probability converging to $1$. For all these vertices we will see the same neighborhood. This also shows that we can impose arbitrary boundary condition, for instance periodic boundary condition means that we consider the sequence of toroidal boxes. We can also consider Aztec diamonds in case of $\mathbb{Z}^2$. Boxes and toroidal boxes will be Benjamini--Schramm convergent even together, and converges to a distribution which is a rooted $\mathbb{Z}^d$ with probability $1$.
\end{Ex}

\begin{Ex} Let $(G_n)$ be a sequence of $d$--regular graphs such that $g(G_n)\to \infty$, where $g(H)$ denotes the girth of a graph $H$, i. e., the length of the shortest cycle in $H$. Then $(G_n)$ Benjamini--Schramm converges to the rooted infinite $d$--regular tree $\mathbb{T}_d$.
\end{Ex}

Now we can present our result. Later we will prove a slightly stronger variant of the following theorem.

\begin{Th} \label{pm-entropy} Let $(G_i)$ be a Benjamini--Schramm convergent sequence of vertex-transitive bipartite $d$--regular graphs. Then the sequence 
$$\frac{\ln \prm(G_i)}{v(G_i)}$$
is convergent.
\end{Th}

Note that in this theorem vertex-transitivity plays the role of the "nice boundary condition". We also note that in case of vertex-transitive graphs, the Benjamini--Schramm convergence simply means that we know larger and larger neighbor of the root of a rooted infinite graph. 
We also would like to point out that a slightly stronger version of Theorem~\ref{entropy-girth} says that if $(G_i)$ is a sequence of bipartite graphs Benjamini--Schramm convergent to the infinite $d$--regular tree, then 
$$\lim_{i\to \infty}\frac{\ln \prm(G_i)}{v(G_i)}=\frac{1}{2}\ln \left(
\frac{(d-1)^{d-1}}{d^{d-2}}\right).$$
So in this case we do not need the vertex-transitivity of the graphs. On the other hand, in \cite{ACFK} the authors gave a sequence of $d$--regular bipartite graphs which are Benjamini--Schramm convergent, still the 
$$\lim_{i\to \infty}\frac{\ln \prm(G_i)}{v(G_i)}$$
does not exist.
\bigskip

It will turn out that the proof of Theorem~\ref{pm-entropy} heavily relies on certain estimate of the smallest zeros of the so-called matching polynomial. This result might be of independent interest of its own.

Let $G$ be a graph on $2n$ vertices, then the matching generating function of $G$ is defined as 
$$M(G,t)=\sum_{k=0}^nm_k(G)t^k=\prod_{k=1}^n (1+\gamma_k(G)t),$$
where $\gamma_1(G)\leq \gamma_2(G)\leq \dots \leq \gamma_n(G)$.
We will prove the following lower bounds for the numbers $\gamma_k(G)$.

\begin{Th} \label{zero-estimation} Let $G$ be a vertex-transitive bipartite $d$--regular graph on $2n$ vertices. Then
$$\gamma_k(G)\geq \frac{d^2}{4(d-1)}\frac{k^2}{n^2}.$$
\end{Th}

This result implies that for a $d$--regular vertex-transitive bipartite graph on $2n$ vertices we have
$$\frac{m_{n-1}(G)}{m_n(G)}=\sum_{k=1}^n \frac{1}{\gamma_k(G)}\leq \sum_{k=1}^n\frac{4(d-1)}{d^2}\frac{n^2}{k^2}\leq \frac{2\pi^2}{3}\frac{(d-1)}{d^2}n^2.$$
On the other hand, one can prove a bit better result:

\begin{Th} \label{ratio} Let $G$ be a $d$--regular vertex-transitive bipartite graph on $2n$ vertices. Then
$$\frac{m_{n-1}(G)}{m_n(G)}\leq \frac{2}{d}n^2.$$
\end{Th}

We mention that the best previous result is due to C. Kenyon, D. Randall, A. Sinclair\footnote{Actually, in the acknowledgment of their paper the authors reveal that this result is due to M. Jerrum.}.

\begin{Th}[C. Kenyon, D. Randall, A. Sinclair, (M. Jerrum)]
Let $G$ be a $d$--regular vertex-transitive graph on $2n$ vertices. If $G$ is bipartite, then
$$\frac{m_{n-1}(G)}{m_n(G)}\leq n^2.$$
If $G$ is not bipartite, then we still have
$$\frac{m_{n-1}(G)}{m_n(G)}\leq 4n^3.$$
\end{Th}

Surprisingly, Theorem~\ref{zero-estimation} and \ref{ratio} fail spectacularly without the vertex-transitivity condition. In Section~\ref{degenerate} we will show that there exist constants $c_d<1$  and $C_d>1$ for which one can construct a graph $G$ with $v(G)=2n$ vertices for arbitrarily large $n$ such that
$$\gamma_1(G)<c_d^{n},$$ 
and  
$$\frac{m_{n-1}(G)}{m_n(G)}>C_d^{n}.$$ 
The construction relies on the one given  in \cite{ACFK}, which used to show that 
$$\lim_{i\to \infty}\frac{\ln \prm(G_i)}{v(G_i)}$$
may not exist for Benjamini--Schramm convergent $d$--regular bipartite graphs.
\bigskip

The rest of the paper is organized as follows. In Section~\ref{entropy-function} we will introduce many important concepts, most notably the entropy function $\lambda_G(p)$, and we establish a few fundamental properties of them. In Section~\ref{main} we will prove Theorem~\ref{pm-entropy}, \ref{zero-estimation} and \ref{ratio}. In Section~\ref{gurvits} we prove Theorem~\ref{stability}.
In Section~\ref{degenerate} we show that vertex-transitivity was indeed crucial in all previous theorems by constructing $d$--regular graphs violating the claims of these theorems.

\section{Preliminaries and basic notions} \label{entropy-function}

Throughout the paper, $G$ denotes a finite graph with vertex set $V(G)$ and
edge set $E(G)$. The number of vertices is denoted by $v(G)$. The \emph{degree} of a vertex
is the number of its neighbors. A graph is called $d$\emph{--regular} if every
vertex has degree exactly $d$. The graph $G-S$ denotes the graph obtained from
$G$ by erasing the vertex set $S$ together with all edges incident to $S$. If $S=\{v\}$ then we simply write $G-v$ instead of $G-\{v\}$.
If $e$ is an edge then $G-e$ denotes the graph with vertex set $V(G)$ and edge set $E(G)\setminus \{e\}$. A \emph{path} $P$ is a sequence of vertices $v_1,v_2,\dots ,v_k$ such that $v_i\neq v_j$ if $i\neq j$ and $(v_i,v_{i+1})\in E(G)$ for $i=1,\dots ,k-1$. A \emph{cycle} $C$ is a 
sequence of vertices $v_1,v_2,\dots ,v_k$ such that $v_i\neq v_j$ if $i\neq j$ and $(v_i,v_{i+1})\in E(G)$ for $i=1,\dots ,k$, where $v_{k+1}=v_1$. The length of the cycle is $k$ in this case. A \emph{$k$--matching} is a set of edges $\{e_1,\dots ,e_k\}$ such that for any $i$ and $j$, the vertex set of $e_i$ and $e_j$ are disjoint, in other words, $e_1,\dots ,e_k$ cover $2k$ vertices together. A \emph{perfect matching} is a matching which covers every vertices.
\bigskip

Let $G=(V,E)$ be a finite graph on $v(G)=2n$ vertices. Let $m_k(G)$ be the number of $k$--matchings ($m_0(G)=1$). Let $t$ be an arbitrary non-negative real number, and
$$M(G,t)=\sum_{k=0}^{n}m_k(G)t^k,$$
and 
$$\mu(G,x)=\sum_{k=0}^{n}(-1)^km_k(G)x^{v(G)-2k}.$$
We call $M(G,t)$ the matching generating function, $\mu(G,x)$ the matching polynomial. Clearly, they encode the same information.
If 
$$M(G,t)=\sum_{k=0}^{n}m_k(G)t^k=\prod_{i=1}^{n}(1+\gamma_i(G) t),$$
then $(\pm \sqrt{\gamma_i(G)})_{i=1}^n$ are the zeros of $\mu(G,x)$. The following fundamental theorem of Heilmann and Lieb \cite{hei} is crucial in all our proofs.

\begin{Th}[Heilmann and Lieb \cite{hei}] \label{Hei} The zeros of the matching polynomial
$\mu(G,x)$ are real, and if the largest degree $D$ is greater than $1$,
then  all zeros lie in the interval $[-2\sqrt{D-1},2\sqrt{D-1}]$. 
\end{Th}

In other words, $\gamma_i$ are real and satisfy the inequality $0\leq \gamma_i(G)\leq 4(D-1)$.
\bigskip

\noindent Let us define
$$p(G,t)=\frac{t\cdot M'(G,t)}{n\cdot  M(G,t)},$$
and 
$$F(G,t)=\frac{\ln M(G,t)}{v(G)}-\frac{1}{2}p(G,t) \ln(t).$$
We will call $p(G,t)$ the density function. Note that there is a natural interpretation of $p(G,t)$. Assume that we choose a random matching $M$ with probability proportional to $t^{|M|}$. Then the expected number of vertices covered by a random matching is $p(G,t)\cdot v(G)$.

\noindent Let 
$$p^*(G)=\frac{2\nu(G)}{v(G)},$$
where $\nu(G)$ denotes the number of edges in the largest matching. If $G$ contains a perfect matching, then clearly $p^*=1$. The function $p=p(G,t)$ is a strictly monotone increasing function which maps $[0,\infty)$ to $[0,p^*)$, where $p^*=p^*(G)$. Hence we can consider its inverse function $t(p)=t(G,p)$ on the interval $[0,p^*)$. Finally, let
$$\lambda_G(p)=F(G,t(p))$$
if $p<p^*$, and $\lambda_G(p)=0$ if $p>p^*$.
Note that we have not defined $\lambda_G(p^*)$ yet. We simply define it as a limit:
$$\lambda_G(p^*)=\lim_{p\nearrow p^*}\lambda_G(p).$$
This limit exists, see part (c) of Proposition~\ref{asymp}.
Later we will extend the definition of $p(G,t), F(G,t)$ and $\lambda_G(p)$ to random rooted graphs $L$.

The intuitive meaning of $\lambda_G(p)$ is the following. Assume that we want to count the number of matchings covering $p$ fraction of the vertices. 
Let us assume that it makes sense: $p=\frac{2k}{v(G)}=\frac{k}{n}$, and so we wish to count $m_k(G)$. Then
$$\lambda_G(p)\approx \frac{\ln m_k(G)}{v(G)}.$$
The more precise formulation of this statement will be given in Proposition~\ref{asymp}. The proof of this proposition is given in the paper \cite{ACH}.
\bigskip

\begin{Prop} \label{asymp} Let $G$ be a finite graph. \\ 
(a) Let $rG$ be $r$ disjoint copies of $G$. Then
$$\lambda_G(p)=\lambda_{rG}(p).$$
(b) If $p<p^*$, then 
$$\frac{d}{dp}\lambda_G(p)=-\frac{1}{2}\ln t(p).$$
(c) The limit 
$$\lim_{p\nearrow p^*}\lambda_G(p)$$
exists. \\
(d) Let $k\leq \nu(G)$ and  $p=\frac{2k}{v(G)}=\frac{k}{n}$. Then
$$\left|\lambda_G(p)-\frac{\ln m_k(G)}{v(G)}\right|\leq \frac{\ln v(G)}{v(G)}.$$
(e) Let us define 
$$\lambda_G(p^*)=\lim_{p\nearrow p^*}\lambda_G(p).$$
Let $k=\nu(G)$, then for $p^*=\frac{2k}{v(G)}$ we have
$$\lambda_G(p^*)=\frac{\ln m_k(G)}{v(G)}.$$
In particular, if $G$ contains a perfect matching, then
$$\lambda_G(1)=\frac{\ln \prm(G)}{v(G)}.$$
(f) If for some function $f(p)$ we have
$$\lambda_G(p)\geq f(p)+o_{v(G)}(1)$$
then
$$\lambda_G(p)\geq f(p).$$
\end{Prop}

\subsection{Benjamini--Schramm convergence and  matching measure} \label{measure}

In this section we review a few things from the paper \cite{ACH}.

\begin{Def} The matching measure of a finite graph is defined as
$$\rho_G=\frac{1}{v(G)}\sum_{z_i:\ \mu(G,z_i)=0}\delta(z_i),$$
where $\delta(s)$ is the Dirac-delta measure on $s$, and we take every $z_i$ into account with its multiplicity.
\end{Def}

In other words, the matching measure  is the probability measure of uniform
distribution on the zeros of $\mu(G,x)$.

\begin{Th}[\cite{ACFK,ACH}] \label{wc} Let $(G_i)$ be  a Benjamini--Schramm convergent bounded degree graph sequence. Let $\rho_{G_i}$ be the matching measure of the graph $G_i$. Then the sequence $(\rho_{G_i})$ is
weakly convergent, i. e., there exists some measure $\rho_{L}$ such that for every bounded continuous function $f$, we have
$$\lim_{i\to \infty} \int f(z)\, d\rho_{G_i}(z)=\int f(z)\, d\rho_{L}(z).$$
\end{Th}

Based on Theorem~\ref{wc}, one can prove the following theorem also proved in \cite{ACH} on limits of $p(G_i,t),t(G_i,p)$ and $\lambda_{G_i}(p)$.

\begin{Th}[\cite{ACH}] \label{entropy} Let $(G_i)$ be a Benjamini--Schramm convergent graph sequence of bounded degree graphs. Then the sequences of functions\\
(a) $$p(G_i,t),$$
(b) $$\frac{\ln M(G_i,t)}{v(G_i)}$$
converge to strictly monotone increasing continuous functions on the interval $[0,\infty)$. \\
Let $p_0$ be a real number between $0$ and $1$ such that $p^*(G_i)\geq p_0$ for all $i$. Then\\
(c) $$t(G_i,p),$$
(d) $$\lambda_{G_i}(p)$$
are convergent for all $0\leq p<p_0$.
\end{Th}

\begin{Def} Let $L$ be a random rooted graph which can be obtained as a limit of Benjamini--Schramm convergent graph sequence $(G_i)$.
Assume that $p^*(G_i)\geq p_0$ for all $i$. Let us define the function $p(L,t),t(L,p)$ and $\lambda_{L}(p)$ as the corresponding limits:
$$p(L,t)=\lim_{i\to \infty}p(G_i,t),\ \ \ t(L,p)=\lim_{i\to \infty}t(G_i,p), \ \ \mbox{and}\ \ \lambda_{L}(p)=\lim_{i\to \infty}\lambda_{G_i}(p),$$
where $t\in [0,\infty)$ and $p\in [0,p_0)$. Finally, let us define
$$\lambda_{L}(p_0)=\lim_{p\nearrow p_0}\lambda_L(p).$$
\end{Def}

\begin{Rem} Clearly, the functions $p(L,t),t(L,p)$ and $\lambda_{L}(p)$ do not depend on the choice of the sequence $(G_i)$ since if $(G_i)$ and $(H_i)$ are two different graph sequences Benjamini--Schramm converging to $L$ then they converge to $L$ even together.

Furthermore, if we can choose the graph sequence $(G_i)$ such that every graph $G_i$ contains a perfect matching then we can choose $p_0$ to be $1$, so we can define $\lambda_L(p)$ on the whole interval $[0,1]$.
\medskip

A simple calculation shows that if $G$ is finite graph then
$$p(G,t)=\int \frac{tz^2}{1+tz^2}\, d\rho_{G}(z)$$
and 
$$F(G,t)=\int \frac{1}{2}\ln\left(1+tz^2\right)\, d\rho_{G}(z)-\frac{1}{2}p(G,t) \ln(t).$$
Now if $(G_i)$ Benjamini--Schramm converges to $L$, then by Theorem~\ref{wc}, the sequence of measures $(\rho_{G_i})$ weakly converges to some measure which we will call $\rho_L$, the matching measure of the random rooted graph $L$. Consequently, for $t>0$, we have
$$p(L,t)=\int \frac{tz^2}{1+tz^2}\, d\rho_{L}(z)$$
and 
$$F(L,t)=\int \frac{1}{2}\ln\left(1+tz^2\right)\, d\rho_{L}(z)-\frac{1}{2}p(L,t) \ln(t).$$
This can be used as an alternative definition for the functions $p(L,t),t(L,p)$ and $\lambda_{L}(p)$. 
\medskip

Note that in general it is not true that
$$\lim_{i\to \infty}\lambda_{G_i}(1)=\lambda_{L}(1).$$
On the other hand, Theorem~\ref{pm-entropy} --the way it is given in Section~\ref{main}, and not in the Introduction-- asserts that it is true if all $G_i$ are vertex-transitive bipartite graphs.
\end{Rem}

\subsection{Inequalities for $t(G,p)$ and $p(G,t)$.} In this part we gather a few facts about the functions $t(G,p)$ and $p(G,t)$.
First, we gather a few facts about $M(G,t)$.

\begin{Lemma} \label{identities} Let $G$ be an arbitrary finite graph. Then\\
(a) $$\sum_{u\in V(G)}M(G-u,t)=v(G)\cdot M(G,t)-2t\cdot M'(G,t)$$ 
(b) $$\sum_{(u,v)\in E(G)}M(G-\{u,v\},t)=M'(G,t).$$
(c) $$M(G,t)M(G-\{u,v\},t)-M(G-u,t)M(G-v,t)=-\sum_{P\in \mathcal{P}_{u,v}}(-t)^{|P|-1}M(G\setminus P,t)^2,$$
where $\mathcal{P}_{u,v}$ is the set of paths connecting the vertices $u$ and $v$. 
\end{Lemma}

Part (a) and (b) are simple double counting. Part (a) appears in the literature (see for instance \cite{god3}) in the form
$$\mu'(G,x)=\sum_{u\in V(G)}\mu(G-u,x).$$
Part (c) is due to Heilmann and Lieb \cite{hei} (see also \cite{god3}) in the form
$$\mu(G-u,x)\mu(G-v,x)-\mu(G,x)\mu(G-\{u,v\},x)=\sum_{P\in \mathcal{P}_{u,v}}\mu(G-P,x)^2.$$

If $G$ is a bipartite graph, then all terms of the right hand side of part (c) have the same signs. This is the key observation why the proofs of Theorem~\ref{stability}, Theorem~\ref{pm-entropy} and Theorem~\ref{zero-estimation} will work. If $G$ is a bipartite graph and $(u,v)\in E(G)$, then there is a trivial term on the right hand side of part (c), namely $tM(G-\{u,v\},t)^2$. Furthermore, in this case all $|P|$ are even, and we can rewrite part (c) as follows.
$$M(G,t)M(G-\{u,v\},t)-M(G-u,t)M(G-v,t)-tM(G-\{u,v\},t)^2=$$
$$=\sum_{P\in \mathcal{P}_{u,v} \atop P\neq (u,v)}t^{|P|-1}M(G\setminus P,t)^2.$$

\begin{Prop} \label{inequalities} (a) Let $G$ be a finite graph with a perfect matching. The function $t(1-p(G,t))$ (or $t(G,p)(1-p)$) is monotone increasing in $t$ (or $p$) and is bounded by a constant $C(G)$ depending on the graph $G$. \\
(b) If $G$ is a $d$--regular finite graph, then
$$p(G,t)\leq \frac{d\cdot t}{1+t}\leq d\cdot t.$$
In case of edge-transitive $d$--regular finite graphs, the inequality can be improved to
$$p(G,t)\leq \frac{d\cdot t}{1+d\cdot t}.$$
(c) If $G$ is a vertex-transitive $d$--regular bipartite graph, then 
$$t(G,p)\leq \frac{p}{d}\left(1-\frac{p}{d}\right)\cdot \frac{1}{(1-p)^2}\leq \frac{d-1}{d^2}\cdot \frac{1}{(1-p)^2}.$$
In fact, with the notation $t=t(G,p)$ we have
$$\frac{p}{d}\left(1-\frac{p}{d}\right)-t(1-p)^2\geq \frac{2}{d\cdot v(G)}\sum_{(u,v)\in E(G)}\left(\sum_{P\in \mathcal{P}_{u,v} \atop P\neq (u,v)}t^{|P|}\frac{M(G\setminus P,t)^2}{M(G,t)^2}\right).$$
Equality holds if $G$ is not only vertex-transitive, but also edge-transitive.
\end{Prop}

\begin{proof} (a) Let us write $M(G,t)$ into the form
$$M(G,t)=\prod_{i=1}^{v(G)/2}(1+\gamma_it),$$
where $\gamma_i$ are positive numbers according to the Heilmann-Lieb theorem.
Then
$$p(G,t)=\frac{2}{v(G)}\sum_{i=1}^{v(G)/2}\frac{\gamma_it}{1+\gamma_it}.$$
Hence
$$t(1-p(G,t))=\frac{2}{v(G)}\sum_{i=1}^{v(G)/2}\frac{t}{1+\gamma_it}.$$
Since all terms of the sum are monotone increasing function of $t$, we see that $t(1-p(G,t))$ is monotone increasing. Furthermore,
$$t(1-p(G,t))=\frac{2}{v(G)}\sum_{i=1}^{v(G)/2}\frac{t}{1+\gamma_it}\leq \frac{2}{v(G)}\sum_{i=1}^{v(G)/2}\frac{1}{\gamma_i}=C(G).$$
\medskip

\noindent (b) By part (b) of Lemma~\ref{identities} we have
$$p(G,t)=\frac{2}{v(G)}\cdot \frac{tM'(G,t)}{M(G,t)}=\frac{2}{v(G)}\cdot \frac{t}{M(G,t)}\sum_{(u,v)\in E(G)}M(G-\{u,v\},t).$$
Next we use the trivial inequality $M(G-\{u,v\},t)\leq M(G-u,t)$. For any edge $(u,v)\in E(G)$, we have
$$(1+t)M(G-\{u,v\},t)\leq M(G-u,t)+tM(G-\{u,v\},t)\leq $$
$$\leq M(G-u,t)+t\sum_{v_i\in N(u)}M(G-\{u,v_i\},t)=M(G,t).$$
Hence 
$$p(G,t)=\frac{2}{v(G)}\cdot \frac{t}{M(G,t)}\sum_{(u,v)\in E(G)}M(G-\{u,v\},t)\leq $$
$$\leq \frac{2}{v(G)}\cdot \frac{t}{M(G,t)} \frac{dv(G)}{2}\frac{M(G,t)}{1+t}=\frac{d\cdot t}{1+t}.$$
If $G$ is edge-transitive then we can use that
$$(1+d\cdot t)M(G-\{u,v\},t)\leq M(G-u,t)+d\cdot tM(G-\{u,v\},t)= $$
$$= M(G-u,t)+t\sum_{v_i\in N(u)}M(G-\{u,v_i\},t)=M(G,t)$$
to obtain
$$p(G,t)\leq \frac{d\cdot t}{1+d\cdot t}.$$
\medskip

\noindent (c) Let us introduce the notation $q=p/d$. For a moment let us assume that the graph $G$ is not only vertex-transitive, but also edge-transitive, so for arbitrary edges $(u,v),(u_1,v_1)\in E(G)$ we have $M(G-\{u,v\},t)=M(G-\{u_1,v_1\},t)$.
Then
$$q=\frac{p}{d}=\frac{1}{d\cdot n}\cdot \frac{tM'(G,t)}{M(G,t)}=\frac{t\cdot M(G-\{u,v\},t)}{M(G,t)}$$
for any edge $(u,v)\in E(G)$ by part (b) of Lemma~\ref{identities}. Furthermore,
$$1-d\cdot q=1-p=\frac{n\cdot M(G,t)-t\cdot M'(G,t)}{n M(G,t)}=\frac{M(G-u,t)}{M(G,t)}$$
for a vertex $u\in V(G)$ by part (a) of Lemma~\ref{identities} using the vertex transitivity.
Hence
$$r=q(1-q)-t(1-d\cdot q)^2=$$
$$=\frac{t(M(G,t)M(G-\{u,v\},t)-t\cdot M(G-\{u,v\},t)^2-M(G-u,t)M(G-v,t))}{M(G,t)^2}=$$
$$=\frac{t}{M(G,t)^2}\left(\sum_{P\in \mathcal{P}_{u,v} \atop P\neq (u,v)}t^{|P|-1}M(G\setminus P,t)^2\right)\geq 0.$$
We can eliminate the edge-transitivity from the argument (but still keeping the vertex-transitivity) if we average the above identity for all edges and we use a Cauchy-Schwarz inequality for the numbers $M(G-\{u,v\},t)$ ($(u,v)\in E(G))$. (The following computation is tedious, but contains no idea.)
$$0\leq \frac{1}{nd}\sum_{(u,v)\in E(G)}\frac{t}{M(G,t)^2}\left(\sum_{P\in \mathcal{P}_{u,v} \atop P\neq (u,v)}t^{|P|-1}M(G\setminus P,t)^2\right)=$$
$$=\frac{1}{nd}\frac{t}{M(G,t)^2}\sum_{(u,v)\in E(G)}\left(M(G,t)M(G-\{u,v\},t)-t\cdot M(G-\{u,v\},t)^2\right)-$$
$$-\frac{1}{nd}\frac{t}{M(G,t)^2}\sum_{(u,v)\in E(G)}M(G-u,t)M(G-v,t)=$$
$$=\frac{1}{nd}\frac{t}{M(G,t)^2}M(G,t)M'(G,t)-\frac{t^2}{M(G,t)^2}\frac{1}{nd}\sum_{(u,v)\in E(G)}M(G-\{u,v\},t)^2-$$
$$-\frac{t}{M(G,t)^2}M(G-u,t)^2\leq $$
$$\leq \frac{1}{nd}\frac{tM'(G,t)}{M(G,t)}-\frac{t^2}{M(G,t)^2}\left(\frac{M'(G,t)}{nd}\right)^2-\frac{t}{n^2\cdot M(G,t)^2}(n\cdot M(G,t)-t\cdot M'(G,t))^2=$$
$$=q(1-q)-t(1-d\cdot q)^2=r.$$
\end{proof}

The following proposition is just a reformulation of the part (c) of Proposition~\ref{inequalities}.

\begin{Prop} \label{tree} Let $G$ be $d$--regular vertex-transitive bipartite graph. Then
$$t(G,p)\leq t(\mathbb{T}_d,p)$$
for $0\leq p<1$ and 
$$p(G,t)\geq p(\mathbb{T}_d,t)$$
for $t\geq 0$.
\end{Prop}

\begin{proof}
It is known (see \cite{csi}) that
$$t(\mathbb{T}_d,p)=\frac{p(d-p)}{d^2(1-p)^2},$$
so the inequality 
$$t(G,p)\leq t(\mathbb{T}_d,p)$$
is just a reformulation of Proposition~\ref{inequalities}.
The other inequality immediately follows from the first one.
We note that
$$p(\mathbb{T}_d,t)=\frac{2d^2t+d-d\cdot\sqrt{1+4(d-1)t}}{2d^2t+2}.$$
\end{proof}

\begin{Rem} The part (b) of Proposition~\ref{inequalities} is only useful for very small values of $t$ and $p$ since if $t\geq 1/d$, then the inequality is trivial.

We would like to point out an interesting dichotomy between finite graphs and infinite lattices. Part (a) shows that 
$$t(G,p)\geq \frac{c}{(1-p)}$$
if $p\geq p_0$ and $c=t(G,p_0)(1-p_0)$, where $p_0$ is an arbitrary positive number. On the other hand,
$$t(G,p)\leq \frac{C(G)}{(1-p)},$$
where
$$C(G)=\frac{2}{v(G)}\sum_{i=1}^{v(G)/2}\frac{1}{\gamma_i}.$$
We mention that if $(G_n)$ converges to an infinite lattice $L$, then the sequence $(C(G_n))$ is not necessarily bounded. So for an infinite lattice $L$, it is not necessarily true that
$$t(L,p)\leq \frac{C(L)}{(1-p)}.$$
In fact, the $d$-regular infinite tree $\mathbb{T}_d$ is already a counterexample.

On the other hand, part (c) of Proposition~\ref{inequalities} and Proposition~\ref{tree} shows that for vertex-transitive $d$--regular bipartite graphs, we have 
$$t(G,p)\leq t(\mathbb{T}_d,p) \leq \frac{d-1}{d^2}\cdot \frac{1}{(1-p)^2}.$$
This shows that if $L$ is the limit of a sequence of $d$--regular vertex-transitive bipartite graphs (like $\mathbb{Z}^d$), then
$$t(L,p)\leq t(\mathbb{T}_d,p)\leq \frac{d-1}{d^2}\cdot \frac{1}{(1-p)^2}.$$
We will prove a matching lower bound for certain random rooted graph (in particular infinite lattices), see Proposition~\ref{dichotomy}. This shows that for infinite lattices, the growth of $t$ can be as fast as $c/(1-p)^2$ unlike in the case of finite graphs.
\end{Rem}

\begin{Prop} \label{dichotomy} Let $L$ be a random rooted graph which can be obtained as a limit of bounded degree finite graphs with perfect matchings. Assume that the measure $\rho_L$ is absolutely continuous to the Lebesgue measure, and has a density function $f(z)$ such that
$$\min_{|z|\leq \varepsilon}f(z)\geq f_0>0$$
for some $\varepsilon$ and $f_0$. Then for $t\geq \frac{1}{\varepsilon^2}$ we have
$$t\geq \frac{f_0^2}{(1-p)^2},$$
where $p=p(L,t)$.
\end{Prop}

\begin{proof}
$$\sqrt{t}(1-p)=\sqrt{t}\int \frac{1}{1+tz^2}\, d\rho_L(z)\geq \sqrt{t}\int_{\{|z|\leq 1/\sqrt{t}\}} \frac{1}{1+tz^2}\, d\rho_L(z)\geq $$
$$\geq \sqrt{t}\int_{\{|z|\leq 1/\sqrt{t}\}} \frac{1}{2}\, d\rho_L(z)\geq \sqrt{t}\cdot \frac{1}{2}\cdot \frac{2}{\sqrt{t}}f_0=f_0.$$
In the last step we have used that for $|z|\leq \frac{1}{\sqrt{t}}\leq \varepsilon$, we have $f(z)\geq f_0$. 
\end{proof}

\begin{Rem} We conjecture that for all $d$, the lattice $\mathbb{Z}^d$ satisfies the condition of the proposition.
\end{Rem} 

\subsection{Vertex-transitivity} By vertex-transitivity we always mean that for every vertex $u$ and $v$, there exists an automorphism $\phi$ of the graph $G$ such that $\phi(u)=v$. In this paper we only use the vertex-transitivity to ensure that
$$M(G-u,t)=M(G-v,t)$$
for every $u$ and $v$. On the other hand, for bipartite graphs there is a natural variant of vertex-transitivity when we only require that the automorphism group of the graph acts transitively on the color classes separately. Apriori this would only give that 
$$M(G-u,t)=M(G-v,t)$$
holds true when $u$ and $v$ belong to the same color class of the bipartite graph. It turns out that for balanced bipartite graphs, this implies that
$$M(G-u,t)=M(G-v,t)$$
for every $u$ and $v$. As a corollary, this weaker variant of the vertex-transitivity can be used everywhere in this paper for $d$--regular bipartite graphs.

\begin{Lemma} \label{balanced} Let $G=(A,B,E)$ be a balanced bipartite graph, i. e., $|A|=|B|$. Then
$$\sum_{u\in A}M(G-u,t)=\sum_{v\in B}M(G-v,t).$$
\end{Lemma}

\begin{proof} Let $\mathcal{M}$ be the set of matchings, and for $M\in \mathcal{M}$, let $|M|$ denote the number of edges in $M$. Then
$$\sum_{u\in A}M(G-u,t)=\sum_{M\in \mathcal{M}}(|A|-|M|)t^{|M|}=\sum_{M\in \mathcal{M}}(|B|-|M|)t^{|M|}=\sum_{v\in B}M(G-v,t).$$
\end{proof}

 Since every $d$--regular bipartite graph is balanced, the following statement is an immediate corollary.
 
\begin{Cor} Let $G=(A,B,E)$ be a $d$--regular bipartite graph such that for every $u,u'\in A$ and $v,v'\in B$ there are automorphisms $\phi_1,\phi_2$ of the graph $G$ such that $\phi_1(u)=u'$ and $\phi_2(v)=v'$. Then for every $u,v\in G$ we have
$$M(G-u,t)=M(G-v,t).$$
\end{Cor}

\section{Perfect matchings of vertex-transitive graphs} \label{main}

In this part we prove Theorem~\ref{pm-entropy}, \ref{zero-estimation} and \ref{ratio}. First we prove Theorem~\ref{zero-estimation}. For sake of convenience we repeat the statement of the theorem with an extra claim showing its connection with the matching measure.
\bigskip

\noindent \textbf{Theorem~\ref{zero-estimation}} \textit{Let $G$ be a vertex-transitive bipartite $d$--regular graph on $2n$ vertices. Then
$$\gamma_k(G)\geq \frac{d^2}{4(d-1)}\frac{k^2}{n^2}.$$
Consequently, for the matching measure $\rho_G$ we have
$$\rho_G([-s,s])\leq \frac{2\sqrt{d-1}}{d}s$$
for all $s\in \mathbb{R}^+$.}
\bigskip

\begin{proof} Recall that for a fix $t$, we have defined 
$$p=p(G,t)=\frac{t\cdot M'(G,t)}{n\cdot  M(G,t)},$$
and in part (c) of Proposition~\ref{inequalities} we have proved that for a vertex-transitive $d$--regular bipartite graph we have
$$t=t(G,p)\leq \frac{p}{d}\left(1-\frac{p}{d}\right)\cdot \frac{1}{(1-p)^2}\leq \frac{d-1}{d^2}\cdot \frac{1}{(1-p)^2}.$$
We will use it in the form
$$t(1-p)^2\leq \frac{d-1}{d^2}.$$
Note that
$$p(G,t)=\frac{1}{n}\sum_{i=1}^{n}\frac{\gamma_it}{1+\gamma_it}.$$
Hence
$$t(1-p)^2=t\left(\frac{1}{n}\sum_{i=1}^{n}\frac{1}{1+\gamma_it}\right)^2\geq t\left(\frac{1}{n}\sum_{i=1}^{k}\frac{1}{1+\gamma_it}\right)^2.$$
Now let $t=\frac{1}{\gamma_k}$, then
$$\frac{d-1}{d^2}\geq t(1-p)^2\geq \frac{1}{\gamma_k}\left(\frac{1}{n}\sum_{i=1}^{k}\frac{1}{1+\frac{\gamma_i}{\gamma_k}}\right)^2\geq \frac{1}{\gamma_k}\left(\frac{k}{2n}\right)^2.$$
In other words,
$$\gamma_k(G)\geq \frac{d^2}{4(d-1)}\frac{k^2}{n^2}.$$
The second claim follows since
$$\rho_G([-s,s])=\frac{1}{2n}|\{ k\ |\ \pm \sqrt{\gamma_k} \in [-s,s]\}|=\frac{1}{n}|\{ k\ |\ \gamma_k\leq s^2\}|.$$
Since
$$\frac{d^2}{4(d-1)}\frac{k^2}{n^2}\leq \gamma_k(G)\leq s^2$$
we have
$$\frac{k}{n}\leq \frac{2\sqrt{d-1}}{d}s.$$
\end{proof}

\begin{Rem}[\cite{ACFK}] Let $(G_i)$ be a sequence of $d$--regular graphs such that $g(G_i)\to \infty$, where $g(H)$ denotes the length of the shortest cycle of a graph $H$. Then $(G_i)$ Benjamini--Schramm converges to the infinite $d$--regular tree $\mathbb{T}_d$. The limit measure $\rho_{\mathbb{T}_d}$ is the Kesten--McKay measure. In general, the matching measure and the spectral measure coincides for (finite and infinite) trees. The  density function of the Kesten--McKay measure is the following
$$f_d(x)=\frac{d \sqrt{4(d-1)-x^2}}{2\pi (d^2-x^2)}\chi_{[-\omega,\omega]},$$
where $\omega=2\sqrt{d-1}$. So the value of the density function at point $0$ is 
$$f_d(0)=\frac{1}{\pi}\cdot \frac{\sqrt{d-1}}{d},$$
this is only multiplicative constant factor away from the bound appearing in Theorem~\ref{zero-estimation}.
\end{Rem}

\noindent \textbf{Theorem~\ref{pm-entropy}} {Let $(G_i)$ be a Benjamini--Schramm convergent sequence of vertex-transitive bipartite $d$--regular graphs. Then the sequence 
$$\lambda_{G_i}(1)=\frac{\ln \prm(G_i)}{v(G_i)}$$
is convergent. Furthermore, if $G_i$ converges to some random rooted graph $L$, then we have
$$\lim_{i\to \infty}  \lambda_{G_i}(1)=\lambda_{L}(1).$$}
\bigskip

\begin{proof} Let $2n_i$ be the number of vertices of the graph $G_i$, and 
$$M(G_i,t)=\sum_{k=0}^{n_i}m_k(G)t^k=\prod_{j=1}^{n_i} (1+\gamma_j(G_i)t),$$
where $\gamma_1(G_i)\leq \gamma_2(G_i)\leq \dots \leq \gamma_{n_i}(G_i)$.
Let $\overline{\rho}_{G_i}$ be the uniform measure on the numbers $\gamma_j(G_i)$, and let $\rho_{G_i}$ be the matching measure of $G_i$. By Theorem~\ref{wc}, the sequence of matching measures $(\rho_{G_i})$ is weakly convergent. This implies that the sequence $(\overline{\rho}_{G_i})$ is weakly convergent too, let $\overline{\rho}_L$ be the limit measure.
Note that Theorem~\ref{zero-estimation} implies that
$$\overline{\rho}_{G_i}([0,t])=\frac{1}{n_i}|\{j \ |\ \gamma_j(G_i)\leq t\}|\leq \frac{2\sqrt{d-1}}{d} \sqrt{t}$$
since 
$$\gamma_j(G_i)\geq \frac{d^2}{4(d-1)}\frac{j^2}{n_i^2}.$$
Because of the weak convergence, this inequality holds for $\overline{\rho}_L$ too. This implies that $\ln(x)$ is uniformly integrable: let $F(t)=\rho([0,t])$ for some measure satisfying the above inequality, and assume $\varepsilon\leq 1$, then integration by parts imply that
$$\left|\int_0^{\varepsilon} \ln(x)\, d\rho(x)\right|=\int_0^{\varepsilon}(-\ln(x))\, dF(x)=$$
$$=\left.F(x)(-\ln(x))\right|_{0}^{\varepsilon}-\int_0^{\varepsilon}F(x)\, d(-\ln(x))\leq F(\varepsilon)\ln\left(\frac{1}{\varepsilon}\right)+\int_0^{\varepsilon}\frac{F(x)}{x}\, dx\leq $$
$$\leq \frac{2\sqrt{d-1}}{d}\left(\sqrt{\varepsilon}\ln\left(\frac{1}{\varepsilon}\right)+\int_0^{\varepsilon}\frac{\sqrt{x}}{x}\, dx\right)=
\frac{2\sqrt{d-1}}{d}\left(\sqrt{\varepsilon}\ln\left(\frac{1}{\varepsilon}\right)+2\sqrt{\varepsilon}\right),$$
which tends to $0$ if $\varepsilon$ tends to $0$. Since
$$\frac{\ln \prm(G_i)}{v(G_i)}=\frac{1}{2}\int \ln(x)\, d\overline{\rho}_{G_i}(x),$$
it immediately implies that 
$$\lim_{i\to \infty}\frac{\ln \prm(G_i)}{v(G_i)}=\frac{1}{2}\int \ln(x)\, d\overline{\rho}_L.$$
\end{proof}

\begin{Cor} Let $(G_i)$ be a Benjamini--Schramm convergent sequence of vertex-transitive bipartite $d$--regular graphs. Let $H_i$ be another Benjamini--Schramm convergent sequence of $d$--regular graphs such that the sequences $(G_i)$ and $(H_i)$ are Benjamini--Schramm convergent together. Then 
$$\limsup_{i\to \infty} \frac{\ln \prm(H_i)}{v(H_i)}\leq \lim_{i\to \infty} \frac{\ln \prm(G_i)}{v(G_i)}.$$
\end{Cor}

Finally, we prove Theorem~\ref{ratio}. For the convenience of the Reader, we repeat the statement.
\bigskip

\noindent \textbf{Theorem~\ref{ratio}} {Let $G$ be a $d$--regular vertex-transitive bipartite graph on $2n$ vertices. Then
$$\frac{m_{n-1}(G)}{m_n(G)}\leq \frac{2}{d}n^2.$$}
\bigskip

\begin{proof} Once again, we use the identity of part (c) of Lemma~\ref{identities}:
$$M(G,t)M(G-\{u,v\},t)-M(G-u,t)M(G-v,t)=-\sum_{P\in \mathcal{P}_{u,v}}(-t)^{|P|-1}M(G\setminus P,t)^2.$$
We apply it for $(u,v)\in E(G)$ again. Then all coeffcients on the right hand side are non-negative. Let us consider the coefficient of $t^{2n-2}$:
$$m_n(G)\cdot m_{n-2}(G-\{u,v\})+m_{n-1}(G)\cdot m_{n-1}(G-\{u,v\})-m_{n-1}(G-u)\cdot m_{n-1}(G-v)\geq 0.$$
Let us use the identity of part (a) of Lemma~\ref{identities} together with the fact that $G$ is vertex-transitive:
$$m_{n-1}(G-u)=m_{n-1}(G-v)=\frac{1}{n}m_{n-1}(G).$$
Hence
$$m_n(G)\cdot m_{n-2}(G-\{u,v\})+m_{n-1}(G)\cdot m_{n-1}(G-\{u,v\})\geq \left(\frac{1}{n}m_{n-1}(G)\right)^2.$$
Now let us sum this inequality for all $(u,v)\in E(G)$ using the fact that
$$\sum_{(u,v)\in E(G)}m_{n-2}(G-\{u,v\})=(n-1)m_{n-1}(G)\ \ \mbox{and} \sum_{(u,v)\in E(G)}m_{n-1}(G-\{u,v\})=n\cdot m_{n}(G).$$
Hence we get that
$$m_n(G)\cdot (n-1)m_{n-1}(G)+m_{n-1}(G)\cdot n\cdot m_{n}(G)\geq nd\cdot \left(\frac{1}{n}m_{n-1}(G)\right)^2.$$
Then
$$\frac{n(2n-1)}{d}\geq \frac{m_{n-1}(G)}{m_n(G)}.$$
\end{proof}

\section{Federbush--expansion and Gurvits's theorem} \label{gurvits}

In this part we prove Theorem~\ref{stability}. As we mentioned in the Introduction,
in \cite{gur2} L. Gurvits proved Friedland's asymptotic lower matching conjecture appearing in \cite{FKM}, which says that if
$G$ is a $d$--regular bipartite graph on $v(G)$ vertices, then
$$\frac{\ln m_k(G)}{v(G)}\geq \G_d(p)+o_{v(G)}(1),$$
where $p=2k/v(G)$. Recall that
$$\G_d(p)= \frac{1}{2}\left(p\ln \left(\frac{d}{p}\right)+(d-p)\ln \left(1-\frac{p}{d}\right)-2(1-p)\ln (1-p)\right).$$
We also noted in the Introduction that there are two inconvenient things in this statement. Namely, the term $o_{v(G)}(1)$, and that $p$ is defined only for special values. It turns out that the two problems are in fact one. If we choose the activity $t$ such that
$$2k/v(G)=p=p(G,t),$$
then
$$\frac{\ln m_k(G)}{v(G)}\approx \lambda_G(p)$$
by part (d) of Proposition~\ref{asymp}.
Hence by part (d) and (f) of Proposition~\ref{asymp}, one can rewrite Gurvits's theorem as follows. (For more detailed explanation, see Section 3 of \cite{csi}.)

\begin{Th}[Gurvits \cite{gur2} (not in this form)] Let $G$ be an arbitrary finite $d$--regular bipartite graph.
Then
$$\lambda_G(p)\geq \G_d(p).$$
\end{Th}

Federbush and his coauthors suggested a related idea developed in a series of papers (see for instance \cite{BFP,Fed1,FF}), namely they suggested to investigate the following expansion:
$$\lambda_G(p)=\frac{1}{2}\left(p\ln \left(\frac{d}{p}\right)-2 (1-p)\ln (1-p)-p+d\sum_{k=2}^{\infty}\frac{a_k}{k(k-1)}\left(\frac{p}{d}\right)^k\right).$$
Comparing Gurvits's theorem and the Federbush-expansion, we see that they differ slightly. Using the Taylor-expansion
$$(1-t)\cdot \ln (1-t)=-t+\sum_{k=2}^{\infty}\frac{1}{k(k-1)}t^k$$
we can see that the following identity holds:
$$p\ln \left(\frac{d}{p}\right)-2(1-p)\ln (1-p)+(d-p)\ln \left(1-\frac{p}{d}\right)=$$
$$=p\ln \left(\frac{d}{p}\right)-2 (1-p)\ln (1-p)-p+d\sum_{k=2}^{\infty}\frac{1}{k(k-1)}\left(\frac{p}{d}\right)^k.$$
This suggests that maybe it would be better to consider the following modified Federbush-expansion:
$$\lambda_G(p)=\frac{1}{2}\left(p\ln \left(\frac{d}{p}\right)+(d-p)\ln \left(1-\frac{p}{d}\right)-2(1-p)\ln (1-p)+d\sum_{k=2}^{\infty}\frac{b_k}{k(k-1)}\left(\frac{p}{d}\right)^k\right).$$
Therefore $b_k=a_k-1$.

It is known that for the $d$--regular infinite tree $\mathbb{T}_d$, we have $b_k\equiv 0$, in other words
$$\lambda_{\mathbb{T}_d}(p)=\G_d(p)$$
It is also known that if $g$ is the length of the shortest non-trivial cycle, then $b_2=b_3=\dots =b_{g-1}=0$. Butera, Federbush and Pernici \cite{BFP} computed the first few elements of $a_k$ for various lattices including $\mathbb{Z}^d$ for small $d$. For instance, for the lattice $\mathbb{Z}^2$ they obtained that $a_2=1,a_3=1,a_4=7,a_5=41,a_6=181,a_7=757,...$. (In other words, $b_2=0,b_3=0,b_4=6,b_5=40,b_6=180,b_7=756,\dots$.) They conjectured that all $a_k$ are positive for $\mathbb{Z}^d$, and it might be true for more general bipartite lattices. We note that the corresponding statement is not true for the $4$-cycle and for the $3$--regular complete bipartite graph on $6$ vertices. This conjecture would imply that if we consider the function $g_d(p)=\lambda_{\mathbb{Z}^d}(p)-\G_{2d}(p)$ then the $k$-th derivative  $g^{(k)}_d(p)$ are non-negative for all $k$. We were not able to settle this conjecture even for $\mathbb{Z}^2$, still it is a very inspiring one. We will show that at least the first derivative is indeed non-negative and it is true in a more general setting.

We will prove  Theorem~\ref{stability}, for sake of convenience we repeat the statement.
\bigskip

\noindent \textbf{Theorem~\ref{stability}.} \textit{Let $G$ be a finite $d$--regular vertex-transitive bipartite graph, where $d\geq 2$. Furthermore, let the gap function $g(p)$ be defined as
$$g(p)=\lambda_G(p)-\G_d(p).$$
Then $g(p)$ is monotone increasing function with $g(0)=0$, in particular $g(p)$ is non-negative. Furthermore, if $G$ contains an $\ell$-cycle, then
$$g(p)\geq \int_0^pf(x)^{\ell}\, dx,$$
where
$$f(x)=\frac{1}{4d}\min(x,(1-x)^2).$$}
\bigskip

\begin{Rem} A bipartite $d$--regular graph always contains a perfect matching, so $p^*=1$ in this case. We also mention that a connected vertex-transitive graph on even number of vertices always contains a perfect matching, while if it has an odd number of vertices then it contains a matching which avoid exactly one vertex.
\end{Rem}

\begin{proof}[Proof of Theorem~\ref{stability}.] As before we use the notation $n=v(G)/2$.

The claim $g(0)=0$ is trivial, so first we only need to prove that $g'(p)\geq 0$.  Let us differentiate the function 
$\lambda_G(p)$ with respect to $p$. By part (b) of Proposition~\ref{asymp} we have
$$\frac{d\lambda_G(p)}{dp}=-\frac{1}{2}\ln(t).$$
On the other hand, by differentiating 
$$\lambda_G(p)=\frac{1}{2}\left(p\ln \left(\frac{d}{p}\right)+(d-p)\ln \left(1-\frac{p}{d}\right)-2(1-p)\ln (1-p)\right)+g(p)$$
with respect to $p$, we get that 
$$-\frac{1}{2}\ln(t)=\frac{d\lambda_G(p)}{dp}=\frac{1}{2}\left(\ln(d)-\ln(p)-\ln \left(1-\frac{p}{d}\right)+2\ln(1-p)\right)+g'(p).$$
Hence
$$g'(p)=\frac{1}{2}\ln \left(\frac{1}{t}\cdot \frac{p}{d}\left(1-\frac{p}{d}\right)\frac{1}{(1-p)^2}\right).$$
Now we immediately see that $g'(p)\geq 0$ by part (c) of Proposition~\ref{inequalities}.

In the next step we prove that short cycles increase the function $g(p)$. As a first step we refine our lower bound for $g'(p)$.
It will be a bit more convenient to carry out the computation if we introduce the notation $q=\frac{p}{d}$. (Note that it is suggested by the Federbush--expansion too.)
Then
$$g'(p)=\frac{1}{2}\ln \left( \frac{q(1-q)}{t\cdot (1-d\cdot q)^2}\right).$$
It is also worth introducing the notation
$$r=q(1-q)-t(1-d\cdot q)^2.$$
Since then
$$g'(p)=\frac{1}{2}\ln \frac{1}{1-\frac{r}{q(1-q)}}\geq \frac{1}{2}\frac{r}{q(1-q)}>\frac{1}{2}d\cdot r.$$
We have seen that $r\geq 0$ as it is  exactly the claim of part (c) of Proposition~\ref{inequalities}:
$$r=\frac{p}{d}\left(1-\frac{p}{d}\right)-t(1-p)^2\geq \frac{2}{d\cdot v(G)}\sum_{(u,v)\in E(G)}\left(\sum_{P\in \mathcal{P}_{u,v} \atop P\neq (u,v)}t^{|P|}\frac{M(G\setminus P,t)^2}{M(G,t)^2}\right)\geq 0.$$
\bigskip

Now we will show that if $G$ contains a cycle of length $\ell$ then
$$g'(p)\geq f(p)^{\ell},$$
where
$$f(x)=\frac{1}{4d}\min(x,(1-x)^2).$$
This will follow from the following inequality:
$$M(G\setminus S,t)(1+d\cdot t)^{|S|}\geq M(G,t).$$
This inequality holds true since every matching of $G$ can be obtained from a matching of $G\setminus S$ plus at most one-one edges incident to every element of $S$. Hence
$$\frac{M(G\setminus S,t)}{M(G,t)}\geq \frac{1}{(1+d\cdot t)^{|S|}}.$$
We will use it to $S=P$, where $P$ is a "short" cycle minus an edge.
Assume that the  length of the shortest cycle is $\ell$. We will call a cycle of size $\ell$ a short cycle. 
Note that every vertex is contained in a short cycle by the vertex-transitivity. This means that at least $2/d$ fraction of the edges are contained in a short cycle, since a cycle goes through two edges at a vertex. Hence
$$g'(p)\geq \frac{1}{2}d\cdot r\geq \frac{1}{2}d\cdot \frac{2}{d}\frac{t^{\ell}}{(1+d\cdot t)^{2\ell}}=\left(\frac{t}{(1+d\cdot t)^2}\right)^{\ell}.$$
Now we bound the function $t/(1+d\cdot t)^2$ according to $d\cdot t\leq 1$ or $d\cdot t>1$. If $d\cdot t\leq 1$, then we use part (b) of Proposition~\ref{inequalities}:
$$\frac{t}{(1+d\cdot t)^2}\geq \frac{t}{4}\geq \frac{p}{4d}.$$
If $d\cdot t>1$, then we use part (c) of Proposition~\ref{inequalities}:
$$\frac{t}{(1+d\cdot t)^2}\geq \frac{t}{(d\cdot t+d\cdot t)^2}=\frac{1}{4d^2}\cdot \frac{1}{t}\geq \frac{1}{4d^2}\cdot \frac{d^2}{d-1}(1-p)^2
\geq \frac{1}{4d}(1-p)^2.$$
Hence
$$\frac{t}{(1+d\cdot t)^2} \geq \frac{1}{4d}\min(p,(1-p)^2).$$
Therefore
$$g(p)\geq \int_0^pf(x)^{\ell}\, dx.$$
\end{proof}

\begin{Rem} Naturally, the statement of Theorem~\ref{stability} remains true for those infinite lattices $L$ which can be obtained as a limit of vertex-transitive bipartite graphs. This is a trivial consequence of Theorem~\ref{entropy}.
\end{Rem}

\begin{Rem} In particular applications, for instance in case of $\mathbb{Z}^3$, it is not really worth using the lower bound
$$g(p)\geq \int_0^p f(x)^{\ell}\, dx.$$
The reason is that one can compute the function $\lambda_L(p)$ quite precisely if $p$ is bounded away from $1$. This can be done exactly the same way as the monomer-dimer entropy was computed in \cite{ACH}. If $p$ is close to $1$, then it is not really easy to compute $\lambda_L(p)$. This is due to the fact that the function $\ln |x|$ is not easy to approximate by polynomials. Still it is useful to compute $g(p)$ with high precision where we can do it, and then use it as a lower bound for $g(1)$. This way we obtain a lower bound for $\lambda_L(1)$.
\end{Rem}

\section{Degenerate graphs} \label{degenerate}

In this part we show that in Theorem~\ref{stability}, \ref{zero-estimation} and \ref{ratio}, the condition vertex-transitivity is indeed necessary in the sense that there are $d$--regular bipartite graphs for which $g'(p_0)<0$ for some $p_0$ unlike in Theorem~\ref{stability}, and $\gamma_1$ is much smaller than in Theorem~\ref{zero-estimation}, and finally the ratio $\frac{m_{n-1}(G)}{m_n(G)}$  can be much bigger than in Theorem~\ref{ratio}.
\bigskip

Given a  finite bipartite $d$--regular graph $G$ and an edge $e$ of $G$, let
$p(e)$ be the probability that a uniform random perfect matching contains $e$. The following theorem was proved in \cite{ACFK}.
The consequence of this theorem was that Theorem~\ref{pm-entropy} is not true without vertex-transitivity.

\begin{Th}[\cite{ACFK}] \label{construction} For any integer $d\ge 3$, there exists a constant $0<c<1$ such that
  for any positive integer $n\geq d$ there exists  a $d$--regular bipartite
  simple graph on $2n$ points with an edge $e$ such that 
	$$p(e)>1-c^n.$$
\end{Th}

Note that for any vertex $v$, we have 
$$\sum_{f: v\in f}p(f)=1.$$
In particular, for an edge $f$ incident to an edge $e$ of Theorem~\ref{construction}, we have $p(f)<c^n$.
\bigskip

Let us introduce
$$s(G)=\frac{m_{n-1}(G)}{n\cdot m_n(G)}.$$
The following proposition is trivial, but important.

\begin{Prop} \label{trivial}
$$\gamma_1\leq \frac{1}{s(G)}\leq n\gamma_1.$$
\end{Prop}

\begin{proof} 
$$s(G)=\frac{1}{n}\sum_{i=1}^n\frac{1}{\gamma_i}.$$
Hence
$$\frac{1}{n \gamma_1}\leq s(G)\leq \frac{1}{\gamma_1}.$$
\end{proof}

\begin{Prop} \label{construction2} Let $G$ be a $d$--regular bipartite graph on $2n$ vertices.  Let $e=(u,v)\in E(G)$, and let $p(e)$ denote the probability that it is contained in a uniform random perfect matching. There exists a bipartite $d$--regular graph $G^*$ on $2(dn+1)$ vertices for which
$$s(G^*)\geq \frac{1}{d(dn+1)}\left(\frac{1}{p(e)}-1\right).$$
\end{Prop}

\begin{proof}
Let us take $d$ copies of $G-e$, and two new vertices $u^*$ and $v^*$. Let us connect  $u^*$ with the vertices corresponding to $v$ in each copy of $G-e$. Similarly, let us connect  $v^*$ with the vertices corresponding to $u$ in each copy of $G-e$. Then the obtained graph $G^*$ is a $d$--regular bipartite graph on $2(dn+1)$ vertices. Note that each perfect matching of $G^*$ consists of an edge pair $(u^*,v_i)$, $(v^*,u_i)$, $d-1$ perfect matchings of $G-e$ and one perfect matching of $G_i-\{u_i,v_i\}$. Hence
$$m_{dn+1}(G^*)=dm_{n-1}(G-\{u,v\})m_n(G-e)^{d-1}.$$
On the other hand,
$$m_{dn}(G^*)\geq m_{dn}(G^*-\{u^*,v^*\})=m_n(G-e)^{d}.$$
Hence,
$$s(G^*)=\frac{1}{dn+1}\cdot \frac{m_{dn}(G^*)}{m_{dn+1}(G^*)}\geq \frac{1}{d(dn+1)}\frac{m_{n}(G-e)}{m_{n-1}(G-\{u,v\})}=$$
$$=\frac{1}{d(dn+1)}\frac{m_{n}(G)-m_{n-1}(G-\{u,v\})}{m_{n-1}(G-\{u,v\})}=\frac{1}{d(dn+1)}\left(\frac{1}{p(e)}-1\right)$$
since 
$$p(e)=\frac{m_{n-1}(G-\{u,v\})}{m_n(G)}.$$
\end{proof}

\begin{Prop} For every integer $d\geq 3$ there exists a sequence of $d$--regular bipartite graphs $(H_i)$ and a constant $c<1$ for which $\gamma_1(H_i)<c^{v(H_i)}$. Furthermore, for every $H_i$ there exists  some $p_0=p_0(H_i)$ such that for the derivative of the gap function $g(p)$, we have $g'(p_0)<0$.
\end{Prop}

\begin{proof} By Theorem~\ref{construction}, there exists a sequence of bipartite $d$--regular graphs $(G_i)$ with some edge $f_i$ for which $p(f_i)<c_1^n$, where $c_1<1$ only depends on $d$. This shows that for the graphs $H_i=G_i^*$ constructed in Proposition~\ref{construction2} we have $s(H_i)>C_1^n$ for large enough $n$, where $C_1>1$ only  depends on $d$.  By Proposition~\ref{trivial} this shows that $\gamma_1(H_i)<c_2^{v(H_i)}$.
\medskip

Let $H=H_i$ on $n$ vertices. We have seen in the proof of Theorem~\ref{stability} that
$$g'(p)=\frac{1}{2}\ln \left( \frac{q(1-q)}{t\cdot (1-d\cdot q)^2}\right),$$
where $q=p/d$. Since $q(1-q)\leq 1$, it is enough to show that for some $t_0$ we have
$$t_0(1-p_0)^2>1.$$
It turns out that in fact $t(1-p)^2$ can be arbitrarily large. Indeed, we have seen in the proof of Theorem~\ref{zero-estimation} that
$$t(1-p)^2=t\left(\frac{1}{n}\sum_{i=1}^{n}\frac{1}{1+\gamma_it}\right)^2\geq \frac{t}{n^2(1+\gamma_1t)^2}.$$
If we choose $t_0=1/\gamma_1$, we see that
$$t_0(1-p_0)^2\geq \frac{1}{4n^2\gamma_1}.$$
Since $\gamma_1$ can be as small as $c^n$, we see that $t_0(1-p_0)^2$ can be arbitrarily large.
\end{proof}

\end{document}